\documentclass{dcds}
\usepackage{enumerate}
\usepackage{multicol}
\usepackage{optidef}
\usepackage{pgfplots}
\pgfplotsset{compat=newest}
\usepgfplotslibrary{fillbetween}
\usetikzlibrary{arrows,calc,decorations.markings,intersections}
\usepackage[colorlinks=true]{hyperref}
\hypersetup{urlcolor=blue, citecolor=red}

\def\currentvolume{X}

\def\ppages{X--XX}

\definecolor{gray}{RGB}{192,192,192} 
\definecolor{black}{RGB}{0,0,0} 

\renewcommand{\labelenumi}{(\alph{enumi})}

\newtheorem{theorem}{Theorem}
\newtheorem{proposition}[theorem]{Proposition}

\newtheorem{corollary}[theorem]{Corollary}

\theoremstyle{definition}

\newtheorem{example}[theorem]{Example}

\newcommand\abs[1]{\left|#1\right|}
\newcommand\lop[1]{\operatorname{LOP}\left(#1\right)}
\newcommand\kt[2]{\operatorname{KT}\left(#1,#2\right)}

\usepackage[pagewise]{lineno}
\begin{document}

\title[Linear Ordering and Rankability]{On the Linear Ordering Problem and the Rankability of Data}
\author[Thomas R. Cameron and Sebastian Charmot and Jonad Pulaj]{}

\email{trc5475@psu.edu}
\email{secharmot@davidson.edu}
\email{jopulaj@davidson.edu}

\begin{abstract}
In 2019, Anderson et al. proposed the concept of rankability, which refers to a dataset's inherent ability to be meaningfully ranked. 
In this article, we give an expository review of the linear ordering problem (LOP) and then use it to analyze the rankability of data.
Specifically, the degree of linearity is used to quantify what percentage of the data aligns with an optimal ranking. 
In a sports context, this is analogous to the number of games that a ranking can correctly predict in hindsight. 
In fact, under the appropriate objective function, we show that the optimal rankings computed via the LOP maximize the hindsight accuracy of a ranking. 
Moreover, we develop a binary program to compute the maximal Kendall tau ranking distance between two optimal rankings, which can be used to measure the diversity among optimal rankings without having to enumerate all optima.
Finally, we provide several examples from the world of sports and college rankings to illustrate these concepts and demonstrate our results. 
\end{abstract}

\maketitle

\centerline{\scshape Thomas R. Cameron}
\medskip
{\footnotesize
 \centerline{Penn State Behrend}
   \centerline{1 Prischak Building}
   \centerline{Erie, PA, 16563, USA}
}%

\medskip

\centerline{\scshape Sebastian Charmot and Jonad Pulaj}
\medskip
{\footnotesize
 \centerline{Davidson College}
   \centerline{P.O. Box 7129}
   \centerline{Davidson, NC, 28035}
}

\bigskip

\section{Introduction}\label{sec:intro}
The linear ordering problem (LOP) is a classical optimization problem that was classified as NP-hard in 1979 by Garey and Johnson~\cite{Garey1979}.
The LOP has received considerable attention in the literature due to its many applications in archaeology,  economics, and ranking~\cite{Chiarni2004,Glover1974,Kondo2014,Leontiff1986,Marti2011}.
Much attention has also been given to the computational aspects of the LOP, which include heuristic, branch-and-bound, and relaxation methods~\cite{Chanas1996,Dantzig1991,Grotschel1984,Grotschel1985:2,Grotschel1985:1,Khachiyan1979,Marti2011,Reinelt1993}.

In many ranking applications, the data is often assumed to be totally rankable, i.e., observations are some noisy evaluations of an objective function that reflects the true unique ranking~\cite{Chau2020}.
In~\cite{Anderson2019},  the authors questioned this assumption by proposing the concept of rankability, which refers to a dataset's inherent ability to be meaningfully ranked. 
When viewed as a digraph, these authors consider an acyclic tournament digraph to be the ideal dataset, which has one unique ranking that perfectly aligns with all pairwise information in the dataset. 
Moreover, given a dataset that can be represented as a digraph with binary weights, they measure rankability by computing the number of edge changes (additions or deletions) that need to be made in order to obtain an acyclic tournament digraph, and the number of distinct acyclic tournament digraphs that can be obtained given that number of edge changes.
The number of edge changes is used as a \emph{distance to ideal}, and the number of distinct acyclic tournament digraphs that can be obtained is used as a \emph{measure of diversity}.
Note that these values can be computed as a LOP, but the latter value requires the enumeration of all optimal solutions, which is not practical for most applications. 
Also, this rankability model is limited to data that can be represented as a digraph with binary weights. 

In this article, we give an expository review of the LOP and then use it to analyze the rankability of data in a manner that is capable of handling weighted data and does not require the enumeration of all optimal solutions. 
Specifically, the degree of linearity is used to quantify what percentage of the data aligns with an optimal ranking, which can be used as the distance to ideal.
In the context of sports, this is analogous to the number of games that a ranking can correctly predict in hindsight. 
In fact, under the appropriate objective function, we show that the optimal rankings computed via the LOP maximize the hindsight accuracy of a ranking. 
Moreover, we develop a binary program to compute the maximal Kendall tau ranking distance between two optimal rankings, which provides a measure of diversity without having to enumerate all optima.
Finally, we provide several examples from the world of sports and college rankings to illustrate these concepts and demonstrate our results. 
\section{The Linear Ordering Problem}\label{sec:lop}
Let $[n]=\{1,2,\ldots,n\}$, where $n\geq 2$, and $a_{ij}\in\mathbb{R}_{\geq 0}$ for $i,j\in [n]$.
Then, the LOP determines a \emph{symmetric re-ordering} of the matrix $A=[a_{ij}]$, i.e., a similarity transformation of $A$ via a permutation matrix, that results in a maximal upper-triangular sum. 
Equivalently, if $A$ is the adjacency matrix of a simple digraph $\Gamma$, then the LOP seeks to find a spanning acyclic tournament sub-digraph of $\Gamma$ with maximal weight sum, where we use the convention that an edge has weight zero if and only if the edge does not exist.
The integer program formulation of this problem, denoted $\lop{A}$, is described as follows~\cite{Marti2011}:
\begin{maxi!}
	{}{\sum_{i\neq j\colon i,j\in [n]}a_{ij}x_{ij}}{}{}\label{eq:lop-obj}
	\addConstraint{x_{ij}+x_{ji}}{=1,\quad\forall i<j\colon i,j\in [n]}\label{eq:lop-const1}
	\addConstraint{x_{ij}+x_{jk}+x_{ki}}{\leq 2,\quad\forall i<j,i<k,j\neq k\colon i,j,k\in [n]}\label{eq:lop-const2}
	\addConstraint{x_{ij}}{\in\{0,1\},\quad\forall i\neq j\colon i,j\in [n],}\label{eq:lop-const3}
\end{maxi!}
where the $x_{ij}$ are decision variables that indicate whether the edge $(i,j)$, with weight $a_{ij}$, should be included in the sub-digraph of $\Gamma$ with maximal weight sum.
Note that constraint~\eqref{eq:lop-const1} ensures that the sub-digraph is a spanning tournament and constraint~\eqref{eq:lop-const2} forces the tournament to contain no dicycles.
Formal definitions of tournaments and dicycles can be found in~\cite{Grotschel1985:1}.

A binary vector $X=[x_{ij}]\in\{0,1\}^{n(n-1)}$, where $x_{ij}$ satisfy constraints~\eqref{eq:lop-const1}--\eqref{eq:lop-const3}, is a \emph{feasible solution} of $\lop{A}$.
We can use constraint~\eqref{eq:lop-const1} to project the feasible solution space to $\{0,1\}^{\binom{n}{2}}$ by substituting every variable $x_{ij}$, where $j>i$, with $1-x_{ji}$.
An \emph{optimal solution} of $\lop{A}$ is a feasible solution of $\lop{A}$ that is maximal with respect to the objective function in~\eqref{eq:lop-obj}.
Given an optimal solution of $\lop{A}$, the value of the objective function is called the \emph{optimal value} of $\lop{A}$.

Let $S_{n}$ denote the set of permutations on $[n]$.
Throughout this article, we use the term permutation and ranking interchangeably. 
Note that there is a bijection between $S_{n}$ and the set of feasible solutions of $\lop{A}$.
Indeed, every feasible solution $X=[x_{ij}]$ corresponds to a unique permutation $\sigma\in S_{n}$, where $x_{ij}=1$ if and only if $\sigma(i)<\sigma(j)$.
We say that $\sigma\in S_{n}$ is an \emph{optimal ranking} of $\lop{A}$ if it corresponds to an optimal solution $X$ of $\lop{A}$.

The \emph{Linear Ordering Polytope}, which we denote by $P^{n}_{\textrm{LO}}$, is the convex hull of all feasible solutions of $\lop{A}$.
While $P^{n}_{\textrm{LO}}$ is technically contained in $\mathbb{R}^{n(n-1)}$, using the constraint~\eqref{eq:lop-const1} it can be projected to $\mathbb{R}^{\binom{n}{2}}$.
For example, see Figure~\ref{fig:lo-polytope}, $P^{3}_{\textrm{LO}}$ can be viewed as the convex hull of the set of permutations on $(1,2,3)$.
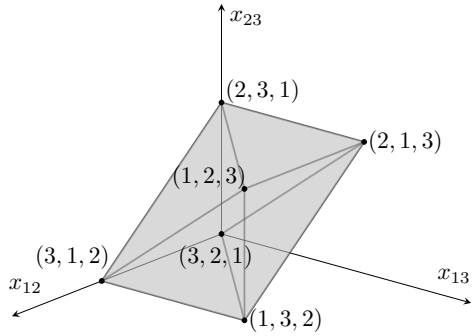
\begin{figure}[ht]
\centering
\resizebox{0.50\textwidth}{!}{
\begin{tikzpicture}
\begin{axis}[
	view={130}{25},
	axis lines=center,
	xtick=\empty, ytick=\empty, ztick=\empty,
	xmin=0,xmax=1.75,ymin=0,ymax=1.75,zmin=0,zmax=1.75,
	xlabel={$x_{12}$},ylabel={$x_{13}$},zlabel={$x_{23}$},
]
\coordinate (r1) at (0,0,0);
\coordinate (r2) at (1,0,0);
\coordinate (r3) at (1,1,0);
\coordinate (r4) at (1,1,1);
\coordinate (r5) at (0,1,1);
\coordinate (r6) at (0,0,1);

\begin{scope}[thick,opacity=0.6]
\draw (r1) -- (r5) -- (r6);
\draw (r1) -- (r3) -- (r5);
\draw (r2) -- (r3) -- (r4);
\draw (r6) -- (r2) -- (r4);
\draw (r6) -- (r4) -- (r5);
\end{scope}

\draw[fill=gray,opacity=0.6,draw=none] (r6) -- (r4) -- (r2);
\draw[fill=gray,opacity=0.6,draw=none] (r2) -- (r4) -- (r3);
\draw[fill=gray,opacity=0.6,draw=none] (r3) -- (r5) -- (r4);
\draw[fill=gray,opacity=0.6,draw=none] (r6) -- (r5) -- (r4);

\draw[black,fill] (0,1,1) circle (1pt) node[xshift=1.75em] {$(2,1,3)$};
\draw[black,fill] (0,0,1) circle (1pt) node[xshift=1.75em,yshift=0.5em] {$(2,3,1)$};
\draw[black,fill] (1,0,0) circle (1pt) node[xshift=-1.25em,yshift=1.0em] {$(3,1,2)$};
\draw[black,fill] (1,1,0) circle (1pt) node[xshift=1.75em] {$(1,3,2)$};
\draw[black,fill] (1,1,1) circle (1pt) node[xshift=-1.5em,yshift=0.5em] {$(1,2,3)$};
\draw[black,fill] (0,0,0) circle (1pt) node[xshift=-0.25em,yshift=-1.0em] {$(3,2,1)$};
\end{axis}
\end{tikzpicture}%
}
\caption{Linear Ordering Polytope: $P^{3}_{\textrm{LO}}$}
\label{fig:lo-polytope}
\end{figure}

As with any polytope, the \emph{minimal equation system} of $P^{n}_{\textrm{LO}}$ is defined as the largest possible collection of linearly independent equations satisfied by all points of $P^{n}_{\textrm{LO}}$~\cite{Grotschel1985:3}.
Furthermore, the \emph{dimension} of $P^{n}_{\textrm{LO}}$, denoted $\dim{P^{n}_{\textrm{LO}}}$, is defined by the cardinality of the largest possible affinely independent subset of $P^{n}_{\textrm{LO}}$.
The dimension theorem states that $\dim{P^{n}_{\textrm{LO}}}$ is equal to the dimension of the ambient space minus the size of the minimal equation system of $P^{n}_{\textrm{LO}}$~\cite[Section 0.5]{Lee2004}.
A classical result from Gr{\"o}tschel, J{\"u}nger and Reinelt~\cite{Grotschel1985:2} shows that the minimal equation system of $P^{n}_{\textrm{LO}}$ is completely described by constraint~\eqref{eq:lop-const1}; hence, the dimension of the linear ordering polytope satisfies
\[
\dim{P^{n}_{\textrm{LO}}} = n(n-1) - \frac{n(n-1)}{2} = \binom{n}{2}.
\]
Note that $a^{T}x\leq a_{0}$, where $a\in\mathbb{R}^{n(n-1)}$ and $a_{0}\in\mathbb{R}$, denotes a \emph{valid inequality} of $P^{n}_{\textrm{LO}}$ if it is satisfied by all $x\in P^{n}_{\textrm{LO}}$. 
Furthermore, a \emph{face} $F\subseteq P^{n}_{\textrm{LO}}$ is defined by $F=\left\{x\in P^{n}_{\textrm{LO}}\colon a^{T}x=a_{0}\right\}$ and \emph{facet} of $P^{n}_{\textrm{LO}}$ is a face of $P^{n}_{\textrm{LO}}$ of dimension $\dim{P^{n}_{\textrm{LO}}} - 1$. 

To solve the LOP, linear programming techniques such as the simplex method~\cite{Dantzig1991} and the ellipsoid method~\cite{Khachiyan1979} require the polytope to satisfy
\[
P^{n}_{\textrm{LO}} \subseteq \left\{x\in\mathbb{R}^{n(n-1)}\colon Cx\leq b\right\},
\]
for some $C\in\mathbb{R}^{m\times n(n-1)}$ and $b\in\mathbb{R}^{m}$.
It is worth noting that the computer software works with the hyperplane description of the polytope, i.e., the right hand side of the above equation, and not the convex hull description, which is not practical as it requires all $n!$ feasible solutions.

For $n\leq 5$, the minimal equation system~\eqref{eq:lop-const1}, the 3-dicycle inequalities~\eqref{eq:lop-const2}, and the trivial inequalities $0\leq x_{ij}\leq 1$, $i\neq j\in[n]$, completely describe $P^{n}_{\textrm{LO}}$; however, for larger dimensions, additional valid inequalities are required. 
To avoid redundant inequalities, there has been considerable interest in defining facet inequalities for $P^{n}_{\textrm{LO}}$.
In~\cite{Grotschel1985:2}, it is shown that simple $k$-fences and simple M\"obius ladders can be used to construct facet defining inequalities of $P^{n}_{\textrm{LO}}$. 
For small dimensions, the Fourier-Motzkin algorithm can be used to compute complete descriptions of $P^{n}_{\textrm{LO}}$. 
For $n=6$, there are $15$ equations forming the minimal equation system, $30$ trivial inequalities, $40$ 3-dicycle inequalities, $120$ $3$-fence inequalities, and $360$ M\"obius ladder inequalities~\cite{Marti2011}.
The complete description for $n=7$ given in~\cite{Reinelt1993} consists of $87472$ facets.

In general, there is good reason to believe it is unlikely to efficiently arrive at a complete description of $P^{n}_{\textrm{LO}}$ such that the number of facets is bound by a polynomial in $n$. 
Indeed, if we had such a complete description of $P^{n}_{\textrm{LO}}$, then we could invoke a polynomial-time algorithm like the ellipsoid method to find the optimal solution. 
In turn this would imply a polynomial-time algorithm for the LOP, which appears unlikely for an NP-hard problem.
\begin{example}\label{ex:digraphs}
Consider the digraphs shown in Figure~\ref{fig:digraphs}, where the weight of each edge is equal to $1$.
Each of these digraphs correspond to a LOP that seeks to find a spanning acyclic tournament sub-digraph with maximal weight sum, using the convention that an edge has weight zero if and only if the edge does not exist. 
\begin{figure}[ht]
\centering
\resizebox{0.80\textwidth}{!}{
\begin{tabular}{cccc}
\begin{tikzpicture}
	\begin{scope}[every node/.style={circle,draw=black,fill=black!20}]
		\node (1) at (2,0) {\textbf{1}};
		\node (2) at (-1,1.732) {\textbf{2}};
		\node (3) at (-1,-1.732) {\textbf{3}};
	\end{scope}
	\begin{scope}[every edge/.style={draw=black,>=latex,->,very thick}]
		\draw (1) edge node{} (2);
		\draw (1) edge node{} (3);
		\draw (2) edge node{} (3);
	\end{scope}
\end{tikzpicture}
&
\begin{tikzpicture}
	\begin{scope}[every node/.style={circle,draw=black,fill=black!20}]
		\node (1) at (2,0) {\textbf{1}};
		\node (2) at (-1,1.732) {\textbf{2}};
		\node (3) at (-1,-1.732) {\textbf{3}};
	\end{scope}
	\begin{scope}[every edge/.style={draw=black,>=latex,->,very thick}]
		\draw (1) edge node{} (2);
		\draw (1) edge node{} (3);
	\end{scope}
	\begin{scope}[every edge/.style={draw=black,>=latex,<->,very thick}]
		\draw (2) edge node{} (3);
	\end{scope}
\end{tikzpicture}
&
\begin{tikzpicture}
	\begin{scope}[every node/.style={circle,draw=black,fill=black!20}]
		\node (1) at (2,0) {\textbf{1}};
		\node (2) at (-1,1.732) {\textbf{2}};
		\node (3) at (-1,-1.732) {\textbf{3}};
	\end{scope}
	\begin{scope}[every edge/.style={draw=black,>=latex,->,very thick}]
		\draw (1) edge node{} (2);
		\draw (2) edge node{} (3);
		\draw (3) edge node{} (1);
	\end{scope}
\end{tikzpicture}
&
\begin{tikzpicture}
	\begin{scope}[every node/.style={circle,draw=black,fill=black!20}]
		\node (1) at (2,0) {\textbf{1}};
		\node (2) at (-1,1.732) {\textbf{2}};
		\node (3) at (-1,-1.732) {\textbf{3}};
	\end{scope}
	\begin{scope}[every edge/.style={draw=black,>=latex,<->,very thick}]
		\draw (1) edge node{} (2);
		\draw (1) edge node{} (3);
		\draw (2) edge node{} (3);
	\end{scope}
\end{tikzpicture}
\\ \textbf{Digraph I} & \textbf{Digraph II} & \textbf{Digraph III} & \textbf{Digraph IV} \\
\end{tabular}%
}
\caption{Simple Digraphs on $3$ vertices}
\label{fig:digraphs}
\end{figure}
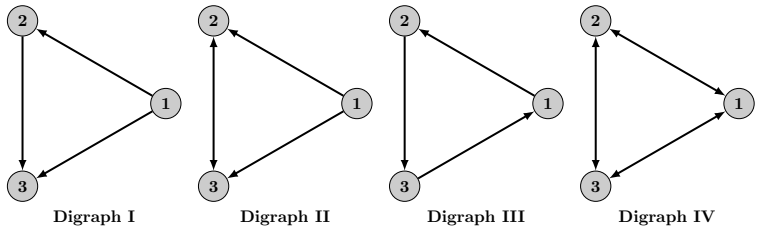

In Figure~\ref{fig:opt-sol}, we display the optimal ranking(s) for each LOP corresponding to the Digraphs I-IV from Figure~\ref{fig:digraphs}.
Each optimal ranking is denoted by a black vertex on the linear ordering polytope; in the case where there are multiple optimal solutions, the black vertices are connected by black lines. 
Note that every point on the black line corresponds to an optimal solution of the linear programming relaxation~\cite{Marti2011}.
In the case where the black lines outline a face of the linear ordering polytope, every point on the face corresponds to an optimal solution of the linear programming relaxation.
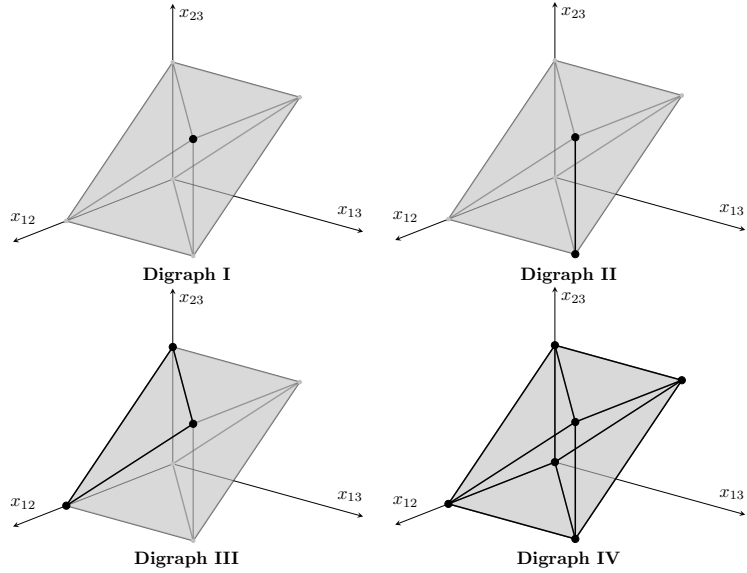
\begin{figure}[ht]
\centering
\resizebox{0.80\textwidth}{!}{
\begin{tabular}{cc}
\begin{tikzpicture}
\begin{axis}[
	view={130}{25},
	axis lines=center,
	xtick=\empty, ytick=\empty, ztick=\empty,
	xmin=0,xmax=1.5,ymin=0,ymax=1.5,zmin=0,zmax=1.5,
	xlabel={$x_{12}$},ylabel={$x_{13}$},zlabel={$x_{23}$},
]
\coordinate (r1) at (0,0,0);
\coordinate (r2) at (1,0,0);
\coordinate (r3) at (1,1,0);
\coordinate (r4) at (1,1,1);
\coordinate (r5) at (0,1,1);
\coordinate (r6) at (0,0,1);

\begin{scope}[thick,opacity=0.6]
\draw (r1) -- (r5) -- (r6);
\draw (r1) -- (r3) -- (r5);
\draw (r2) -- (r3) -- (r4);
\draw (r6) -- (r2) -- (r4);
\draw (r6) -- (r4) -- (r5);
\end{scope}

\draw[fill=gray,opacity=0.6,draw=none] (r6) -- (r4) -- (r2);
\draw[fill=gray,opacity=0.6,draw=none] (r2) -- (r4) -- (r3);
\draw[fill=gray,opacity=0.6,draw=none] (r3) -- (r5) -- (r4);
\draw[fill=gray,opacity=0.6,draw=none] (r6) -- (r5) -- (r4);

\draw[gray,fill] (0,1,1) circle (1pt);
\draw[gray,fill] (0,0,1) circle (1pt);
\draw[gray,fill] (1,0,0) circle (1pt);
\draw[gray,fill] (1,1,0) circle (1pt);
\draw[black,fill] (1,1,1) circle (2pt);
\draw[gray,fill] (0,0,0) circle (1pt);
\end{axis}
\end{tikzpicture}
&
\begin{tikzpicture}
\begin{axis}[
	view={130}{25},
	axis lines=center,
	xtick=\empty, ytick=\empty, ztick=\empty,
	xmin=0,xmax=1.5,ymin=0,ymax=1.5,zmin=0,zmax=1.5,
	xlabel={$x_{12}$},ylabel={$x_{13}$},zlabel={$x_{23}$},
]
\coordinate (r1) at (0,0,0);
\coordinate (r2) at (1,0,0);
\coordinate (r3) at (1,1,0);
\coordinate (r4) at (1,1,1);
\coordinate (r5) at (0,1,1);
\coordinate (r6) at (0,0,1);

\begin{scope}[thick,opacity=0.6]
\draw (r1) -- (r5) -- (r6);
\draw (r1) -- (r3) -- (r5);
\draw (r2) -- (r3) -- (r4);
\draw (r6) -- (r2) -- (r4);
\draw (r6) -- (r4) -- (r5);
\end{scope}

\draw[fill=gray,opacity=0.6,draw=none] (r6) -- (r4) -- (r2);
\draw[fill=gray,opacity=0.6,draw=none] (r2) -- (r4) -- (r3);
\draw[fill=gray,opacity=0.6,draw=none] (r3) -- (r5) -- (r4);
\draw[fill=gray,opacity=0.6,draw=none] (r6) -- (r5) -- (r4);

\draw[gray,fill] (0,1,1) circle (1pt);
\draw[gray,fill] (0,0,1) circle (1pt);
\draw[gray,fill] (1,0,0) circle (1pt);
\draw[black,fill] (1,1,0) circle (2pt);
\draw[black,fill] (1,1,1) circle (2pt);
\draw[gray,fill] (0,0,0) circle (1pt);

\draw[draw=black,thick] (r3)--(r4);
\end{axis}
\end{tikzpicture}
\\ \textbf{Digraph I} & \textbf{Digraph II} \\
\begin{tikzpicture}
\begin{axis}[
	view={130}{25},
	axis lines=center,
	xtick=\empty, ytick=\empty, ztick=\empty,
	xmin=0,xmax=1.5,ymin=0,ymax=1.5,zmin=0,zmax=1.5,
	xlabel={$x_{12}$},ylabel={$x_{13}$},zlabel={$x_{23}$},
]
\coordinate (r1) at (0,0,0);
\coordinate (r2) at (1,0,0);
\coordinate (r3) at (1,1,0);
\coordinate (r4) at (1,1,1);
\coordinate (r5) at (0,1,1);
\coordinate (r6) at (0,0,1);

\begin{scope}[thick,opacity=0.6]
\draw (r1) -- (r5) -- (r6);
\draw (r1) -- (r3) -- (r5);
\draw (r2) -- (r3) -- (r4);
\draw (r6) -- (r2) -- (r4);
\draw (r6) -- (r4) -- (r5);
\end{scope}

\draw[fill=gray,opacity=0.6,draw=none] (r6) -- (r4) -- (r2);
\draw[fill=gray,opacity=0.6,draw=none] (r2) -- (r4) -- (r3);
\draw[fill=gray,opacity=0.6,draw=none] (r3) -- (r5) -- (r4);
\draw[fill=gray,opacity=0.6,draw=none] (r6) -- (r5) -- (r4);

\draw[gray,fill] (0,1,1) circle (1pt);
\draw[black,fill] (0,0,1) circle (2pt);
\draw[black,fill] (1,0,0) circle (2pt);
\draw[gray,fill] (1,1,0) circle (1pt);
\draw[black,fill] (1,1,1) circle (2pt);
\draw[gray,fill] (0,0,0) circle (1pt);

\draw[draw=black,thick] (r4) -- (r2) -- (r6) -- cycle;
\end{axis}
\end{tikzpicture}
&
\begin{tikzpicture}
\begin{axis}[
	view={130}{25},
	axis lines=center,
	xtick=\empty, ytick=\empty, ztick=\empty,
	xmin=0,xmax=1.5,ymin=0,ymax=1.5,zmin=0,zmax=1.5,
	xlabel={$x_{12}$},ylabel={$x_{13}$},zlabel={$x_{23}$},
]
\coordinate (r1) at (0,0,0);
\coordinate (r2) at (1,0,0);
\coordinate (r3) at (1,1,0);
\coordinate (r4) at (1,1,1);
\coordinate (r5) at (0,1,1);
\coordinate (r6) at (0,0,1);

\begin{scope}[thick,opacity=0.6]
\draw (r1) -- (r5) -- (r6);
\draw (r1) -- (r3) -- (r5);
\draw (r2) -- (r3) -- (r4);
\draw (r6) -- (r2) -- (r4);
\draw (r6) -- (r4) -- (r5);
\end{scope}

\draw[fill=gray,opacity=0.6,draw=none] (r6) -- (r4) -- (r2);
\draw[fill=gray,opacity=0.6,draw=none] (r2) -- (r4) -- (r3);
\draw[fill=gray,opacity=0.6,draw=none] (r3) -- (r5) -- (r4);
\draw[fill=gray,opacity=0.6,draw=none] (r6) -- (r5) -- (r4);

\draw[black,fill] (0,1,1) circle (2pt);
\draw[black,fill] (0,0,1) circle (2pt);
\draw[black,fill] (1,0,0) circle (2pt);
\draw[black,fill] (1,1,0) circle (2pt);
\draw[black,fill] (1,1,1) circle (2pt);
\draw[black,fill] (0,0,0) circle (2pt);

\draw[draw=black,thick] (r1) -- (r5) -- (r6) -- cycle;
\draw[draw=black,thick] (r1) -- (r2) -- (r6);
\draw[draw=black,thick] (r2) -- (r3) -- (r4) -- cycle;
\draw[draw=black,thick] (r1) -- (r3) -- (r5);
\draw[draw=black,thick] (r6) -- (r4) -- (r5);
\end{axis}
\end{tikzpicture}
\\ \textbf{Digraph III} & \textbf{Digraph IV} \\
\end{tabular}%
}
\caption{Optimal Solution(s) on Linear Ordering Polytope}
\label{fig:opt-sol}
\end{figure}

Since Digraph I is itself an acyclic tournament, it follows that there is one unique optimal solution corresponding to the ranking $(1,2,3)$.
There is a bidirectional edge between vertices $2$ and $3$ in Digraph II.
Hence, the objective value in~\eqref{eq:lop-obj} does not vary between solutions $x_{23}=1$ and $x_{32}=1$. 
Therefore, there are two distinct optimal solutions corresponding to the optimal rankings $(1,2,3)$ and $(1,3,2)$. 
Digraph III is a 3-dicycle and is therefore isomorphic under cyclic-permutations of the vertices.
Thus, there are three distinct optimal solutions corresponding to the optimal rankings $(1,2,3)$, $(3,1,2)$, and $(2,3,1)$. 
Finally, Digraph IV is the complete digraph on $3$ vertices, and it follows that every point on $P^{3}_{\textrm{LO}}$ corresponds to an optimal ranking.
\end{example}

\section{Rankability and the Linear Ordering Problem}\label{sec:rank-lop}
In ranking applications, the data is often assumed to be totally rankable, i.e., observations are some noisy evaluations of an objective function that reflects the true unique ranking. 
Recently, however, some in the ranking community have challenged this assumption; specifically, in~\cite{Anderson2019,Cameron2020_SR} quantitative metrics are proposed for measuring the rankability of data, and in~\cite{Chau2020} general preference modeling methods are developed without assuming total rankability.

Since the linear ordering problem results in a ranking that is optimal with respect to the objective function in~\eqref{eq:lop-obj}, it is natural to consider what properties of the LOP can be used to analyze the rankability of the underlying data.
In particular, we show that the degree of linearity can be used to quantify what percentage of the data aligns with an optimal ranking, and the optimal diameter can be used to quantify the diversity among all optimal ranking(s).

\subsection{The Degree of Linearity}\label{subsec:deg-lin}
Given $a_{ij}\in\mathbb{R}_{\geq 0}$, for $i,j\in[n]$, not all zero, the \emph{degree of linearity} of $A=[a_{ij}]$ is defined as follows~\cite{Marti2011}:
\[
\lambda(A) = \max_{\sigma\in S_{n}}\frac{\sum_{\sigma(i)<\sigma(j)}a_{ij}}{\sum_{i\neq j}a_{ij}},
\]
i.e., $\lambda(A)$ is the maximum upper-triangular sum as a percentage of the total sum of $A$, over all possible symmetric re-orderings of $A$.
Since the maximum is clearly attained at an optimal ranking of $\lop{A}$, it follows that the degree of linearity of $A$ can be written as
\[
\lambda(A) = \frac{k^{*}}{\sum_{i\neq j}a_{ij}},
\]
where $k^{*}$ denotes the optimal value of $\lop{A}$.
Note that Digraphs I--IV in Figure~\ref{fig:digraphs} have an adjacency matrix with degree of linearity equal to $1$, $3/4$, $2/3$, and $1/2$, respectively. 
In general, the degree of linearity is a value between $1/2$ and $1$~\cite{Marti2011}.

The degree of linearity is used in linear programming to study the integrality gap of the tournament relaxation~\cite{Marti2011}, and in economics to study economic development~\cite{Chiarni2004}.
With regards to the former application, note that the \emph{integrality gap} is the ratio of the optimal objective value associated with the linear programming relaxation and the optimal objective value of the LOP.
For the latter application, researchers have noted that large and highly developed economies have a low degree of linearity since there is a high circulation in the flow of goods among sectors, whereas underdeveloped economies tend to exhibit a clear hierarchy and, hence, a high degree of linearity.
Typical degree of linearity values are $70\%$ for developed economies and $90\%$ for underdeveloped economies~\cite{Leontiff1986}. 

For rankability purposes, we are interested in the degree of linearity as it quantifies the percentage of data that aligns with an optimal ranking.
\begin{example}\label{ex:college-rankings}
\begin{figure}
\centering
\resizebox{1.0\textwidth}{!}{
\begin{tabular}{c|c|c|c|c|c}
College & Selectivity & Faculty Resource & Student/Faculty Ratio & Graduate Retention & Financial Resources \\
\hline
Amherst & 5 & 7 & 9/1 & 1 & 10 \\
Bowdoin & 8 & 14 & 10/1 & 6 & 14 \\
Carleton & 12 & 16 & 9/1 & 4 & 27 \\
Claremont & 14 & 4 & 9/1 & 11 & 21 \\
Haveford & 2 & 5 & 8/1 & 6 & 15 \\
Middlebury & 6 & 17 & 9/1 & 11 & 3 \\
Pomona & 2 & 20 & 8/1 & 1 & 6 \\
Swarthmore & 6 & 7 & 8/1 & 4 & 9 \\
Wellesley & 12 & 12 & 8/1 & 14 & 10 \\
Williams & 4 & 3 & 7/1 & 1 & 6 \\
\hline
\end{tabular}%
}
\caption{College Features from U.S. World News 2013 Rankings}
\label{fig:college-ranking}
\end{figure}

In Figure~\ref{fig:college-ranking}, the rankings of several features used in the U.S. World News College Rankings are reported for 10 liberal arts colleges, listed in alphabetical order, for the year 2013. 
These rankings, and others, can be found from the supplementary material in~\cite{Chartier2014}.
Numbering the colleges as they appear from top to bottom in Figure~\ref{fig:college-ranking}, i.e., in alphabetical order, let $a_{ij}$ denote the number of features that college $i$ outperforms college $j$, where $0.5$ is awarded to both schools in the case of a tie. 
Then, the adjacency matrix $A=[a_{ij}]$ is defined by
\[
A = \begin{bmatrix}
		0 & 5 & 4.5 & 3.5 & 2 & 3.5 & 1.5 & 2.5 & 3.5 & 0.5 \\
		0 & 0 & 3 & 3 & 1.5 & 2 & 1 & 0 & 2 & 0 \\
		0.5 & 2 & 0 & 2.5 & 1 & 2.5 & 1 & 0.5 & 1.5 & 0 \\
		1.5 & 2 & 2.5 & 0 & 1 & 2 & 1 & 1 & 2 & 0 \\
		3 & 3.5 & 4 & 4 & 0 & 4 & 2 & 2.5 & 3.5 & 1 \\
		1.5 & 3 & 2.5 & 3 & 1 & 0 & 1 & 1.5 & 3 & 0 \\ 
		3.5 & 4 & 4 & 4 & 3 & 4 & 0 & 3.5 & 3.5 & 2 \\
		2.5 & 5 & 4.5 & 4 & 2.5 & 3.5 & 1.5 & 0 & 4.5 & 0 \\
		1.5 & 3 & 3.5 & 3 & 1.5 & 2 & 1.5 & 0.5 & 0 & 0 \\
		4.5 & 5 & 5 & 5 & 4 & 5 & 3 & 5 & 5 & 0 \\
	\end{bmatrix},
\]
and an optimal solution of $\lop{A}$ will correspond to an optimal aggregate of the college feature rankings in Figure~\ref{fig:college-ranking}.  
For instance, using the \emph{CPLEX} optimization suite~\cite{cplex2019}, we are able to find an optimal ranking
\[
\sigma = (10,7,8,5,1,6,9,2,4,3)
\]
of $\lop{A}$, which indicates that with respect to the given data, Williams can be considered the best college and Carleton the worst college.
Moreover, the symmetric re-ordering of $A$ with respect to $\sigma$ results in the following matrix
\[
A_{\sigma} = \begin{bmatrix}
		0 &  3 &  4 &  4.5 &  5 &  5 &  5 &  5 &  5 &  5 \\
2 &  0 &  3 &  3.5 &  3.5 &  4 &  3.5 &  4 &  4 &  4 \\
1 &  2 &  0 &  3 &  2.5 &  4 &  3.5 &  3.5 &  4 &  4 \\
0.5 &  1.5 &  2 &  0 &  2.5 &  3.5 &  3.5 &  5 &  4.5 &  3 \\
0 &  1.5 &  2.5 &  2.5 &  0 &  3.5 &  4.5 &  5 &  4.5 &  4 \\
0 &  1 &  1 &  1.5 &  1.5 &  0 &  3 &  3 &  2.5 &  3 \\
0 &  1.5 &  1.5 &  1.5 &  0.5 &  2 &  0 &  3 &  3.5 &  3 \\ 
0 &  1 &  1.5 &  0 &  0 &  2 &  2 &  0 &  3 &  3 \\
0 &  1 &  1 &  0.5 &  0.5 &  2.5 &  1.5 &  2 &  0 &  2 \\ 
0 &  1 &  1 &  1.5 &  1 &  2 &  2 &  2 &  2.5 &  0 \\
	\end{bmatrix},
\]
which can be formed by swapping the rows and columns indicated by the permutation $\sigma$.
Taking the ratio of the upper-triangular sum and the total matrix sum gives
\[
\lambda(A) = \frac{169}{225} \approx 0.75.
\]
It is important to note that the values in the upper-triangular portion of $A_{\sigma}$ correspond to the number of features for which a higher ranked college in $\sigma$ outperformed a lower ranked college.
Thus, approximately $75\%$ of the data aligns with the optimal ranking $\sigma$, i.e., only $25\%$ of the pairwise feature comparisons between any two colleges will disagree with the optimal ranking $\sigma$. 

Furthermore, this observation holds for all optimal rankings of $\lop{A}$.
That is, while different optimal rankings will disagree on the ordering of certain schools, e.g., see Example~\ref{ex:kt_college_rankings}, they all align with approximately $75\%$ of the pairwise data.
\end{example}
\subsection{The Optimal Diameter}\label{subsec:opt-dia}
Feasible binary programs such as the linear ordering problem and traveling salesman problem, often have multiple optimal solutions, which is of interest in applications as they allow the user to choose between alternative optima without deteriorating the objective function~\cite{Glover1998,Kuo1993,Petit2019,Tsai2008}.
Note that the Digraphs I--IV in Figure~\ref{fig:digraphs} correspond to LOPs with $1$, $2$, $3$, and $6$ optimal solutions, respectively.

In applications, the number of optimal solutions can be too large for enumeration to be practical. 
Therefore, in~\cite{Cameron2020_KT}, the authors present the optimal diameter of a feasible binary program as a metric for measuring the diversity among all optimal solutions. 
In particular, a binary program is developed whose optima contain two optimal solutions of the given binary program that are as diverse as possible with respect to the optimal diameter. 
The main focus of their study is on the diameter polytope, i.e., the polytope underlying the optimal diameter binary program, and to show that much of its structure is inherited from the polytope underlying the given binary program. 

For rankability purposes, we are not interested in studying the diameter polytope of general binary programs, but rather the optimal diameter of a given LOP.
Also, in this context, the optimal diameter can be regarded as the maximal Kendall tau rank distance between any two optimal rankings of the LOP~\cite[Proposition 3.2]{Cameron2020_KT}.
The \emph{Kendall tau rank distance} was developed by Maurice Kendall in 1938~\cite{Kendall1938}, and is equal to the number of discordant pairs between two rankings.
For instance, the Digraphs I-IV in Figure~\ref{fig:digraphs} correspond to LOPs whose optimal solution(s) have a maximal Kendall tau ranking distance of $0$, $1$, $2$, and $3$, respectively. 

Here, we develop a binary program whose optima contain two optimal solutions of a given LOP that are as far apart as possible with respect to the Kendall tau ranking distance. 
This program is similar to but more straightforward than the optimal diameter binary program in~\cite{Cameron2020_KT}.
Also, this program is similar to those developed in~\cite{Anderson2021,Kondo2014}, where the author in~\cite{Kondo2014} is interested in two optimal solutions of separate LOPs that are as close together as possible, and the authors in~\cite{Anderson2021} are interested in two optimal solutions of the same LOP that are as far apart as possible. 
The programs in~\cite{Anderson2021,Kondo2014} use the projection trick for LOP discussed in Section~\ref{sec:lop} to cut the number variables in half. 
After implementing the same projection trick, our model will have $\binom{n}{2}$ fewer variables and $\binom{n}{2}$ fewer constraints than the models in~\cite{Anderson2021,Kondo2014}.
Finally, we note that the validity of the models in~\cite{Anderson2021,Kondo2014} is never proven as we do for our model here.

Let $k^{*}$ denote the optimal value of $\lop{A}$ and consider the following binary program, which we denote by $\kt{A}{k^{*}}$:
\begin{mini!}
	{}{\sum_{i\neq j\colon i,j\in[n]}z_{ij}}{}{}\label{eq:kt-obj}
	\addConstraint{x_{ij}+x_{ji}}{=1,\quad\forall i<j\colon i,j\in[n]}\label{eq:kt-constx1}
	\addConstraint{y_{ij}+y_{ji}}{=1,\quad\forall i<j\colon i,j\in[n]}\label{eq:kt-consty1}
	\addConstraint{x_{ij}+x_{jk}+x_{ki}}{\leq 2,\quad\forall i<j,i<k,j\neq k\colon i,j,k\in[n]}\label{eq:kt-constx2}
	\addConstraint{y_{ij}+y_{jk}+y_{ki}}{\leq 2,\quad\forall i<j,i<k,j\neq k\colon i,j,k\in[n]}\label{eq:kt-consty2}
	\addConstraint{\sum_{i\neq j\colon i,j\in[n]}a_{ij}x_{ij}}{=k^{*}}\label{eq:kt-constx3}
	\addConstraint{\sum_{i\neq j\colon i,j\in[n]}a_{ij}y_{ij}}{=k^{*}}\label{eq:kt-consty3}
	\addConstraint{x_{ij}+y_{ij}-z_{ij}}{\leq 1,\quad\forall i\neq j\colon i,j\in[n]}\label{eq:kt-const}
	\addConstraint{x_{ij},y_{ij},z_{ij}}{\in\{0,1\},\quad\forall i\neq j\colon i,j\in[n],}\label{eq:kt-constxyz}
\end{mini!}
where the $x_{ij}$ and $y_{ij}$ are decision variables that form optimal solutions of $\lop{A}$, by constraints~\eqref{eq:kt-constx1}--\eqref{eq:kt-consty3}, and the decision variable $z_{ij}$ is used to ensure that these optimal solutions are as far apart as possible, by constraint~\eqref{eq:kt-constxyz}.

A binary vector $X\oplus Y\oplus Z$, where $X=[x_{ij}]$, $Y=[y_{ij}]$, and $Z=[z_{ij}]$, is a \emph{feasible solution} of $\kt{A}{k^{*}}$ provided that $x_{ij},~y_{ij},~z_{ij}$ satisfy constraints~\eqref{eq:kt-constx1}--\eqref{eq:kt-constxyz}.
An \emph{optimal solution} of $\kt{A}{k^{*}}$ is a feasible solution of $\kt{A}{k^{*}}$ that is minimal with respect to the objective function in~\eqref{eq:kt-obj}. 
Given an optimal solution of $\kt{A}{k^{*}}$, the objective function value is called the \emph{optimal value} of $\kt{A}{k^{*}}$.
\begin{proposition}\label{prop:kt-opt}
Given an optimal solution of $\kt{A}{k^{*}}$, $z_{ij}=1$ for some $i\neq j$, $i,j\in[n]$, if and only if $x_{ij}=y_{ij}=1$. 
\end{proposition}
\begin{proof}
For the sake of contradiction, suppose that $z_{ij}=1$ for some $i\neq j$, $i,j\in[n]$, and either $x_{ij}=0$ or $y_{ij}=0$. 
Then, constraint~\eqref{eq:kt-const} would also be satisfied if $z_{ij}=0$; hence, $z_{ij}=1$ contradicts the objective sense of~\eqref{eq:kt-obj}.
Conversely, suppose that $x_{ij}=y_{ij}=1$ for some $i\neq j$, $i,j\in[n]$.
Then, $z_{ij}=1$ in order to satisfy~\eqref{eq:kt-const}, and the result follows. 
\end{proof}

For two rankings, $\sigma_{1},\sigma_{2}\in S_{n}$, we define the set of \emph{discordant pairs}, denoted $D(\sigma_{1},\sigma_{2})$, by the set of all $(i,j)\in[n]\times[n]$ such that $\sigma_{1}$ and $\sigma_{2}$ don't agree on their relative ranking of $i$ and $j$, i.e., $\sigma_{1}(i)<\sigma_{1}(j)$ and $\sigma_{2}(i)>\sigma_{2}(j)$, or $\sigma_{1}(i)>\sigma_{1}(j)$ and $\sigma_{2}(i)<\sigma_{2}(j)$. 
Also, the set of \emph{concordant pairs}, denoted $C(\sigma_{1},\sigma_{2})$, is defined by the set of $(i,j)\in[n]\times[n]$ such that $\sigma_{1}$ and $\sigma_{2}$ do agree on their relative ranking of $i$ and $j$, i.e., $C(\sigma_{1},\sigma_{2}) = T - D(\sigma_{1},\sigma_{2})$, where $T=\left\{(i,j)\in[n]\times[n]\colon~i<j\right\}$.
Note that the Kendall tau ranking distance between $\sigma_{1}$ and $\sigma_{2}$ is equal to $\abs{D(\sigma_{1},\sigma_{2})}$. 
\begin{theorem}\label{thm:kt-opt}
Let $\sigma_{1},\sigma_{2}\in S_{n}$ be optimal rankings for $\lop{A}$ that correspond to an optimal solution of $\kt{A}{k^{*}}$.
Then, the optimal value of $\kt{A}{k^{*}}$ is equal to $\abs{C(\sigma_{1},\sigma_{2})}$.
\end{theorem}
\begin{proof}
By Proposition~\ref{prop:kt-opt}, the optimal value of $\kt{A}{k^{*}}$ counts the number of distinct pairs $(i,j)\in[n]\times[n]$, where $i\neq j$ and $x_{ij}=y_{ij}=1$.
If $i<j$, then $x_{ij}=y_{ij}=1$ if and only if $(i,j)\in C(\sigma_{1},\sigma_{2})$; otherwise, if $i>j$, then $x_{ij}=y_{ij}=1$ if and only if $(j,i)\in C(\sigma_{1},\sigma_{2})$.
Therefore, the optimal value of $\kt{A}{k^{*}}$ counts the number of distinct pairs in $C(\sigma_{1},\sigma_{2})$. 
\end{proof}

Now, given $\sigma_{1},\sigma_{2}\in S\subseteq S_{n}$, we say that $D(\sigma_{1},\sigma_{2})$ is \emph{maximal} with respect to $S$ provided that $\abs{D(\pi_{1},\pi_{2})}\leq \abs{D(\sigma_{1},\sigma_{2})}$ for all $\pi_{1},\pi_{2}\in S$. 
Similarly, we say that $C(\sigma_{1},\sigma_{2})$ is minimal with respect to $S$ provided that $\abs{C(\sigma_{1},\sigma_{2})}\leq\abs{C(\pi_{1},\pi_{2})}$ for all $\pi_{1},\pi_{2}\in S$. 
\begin{corollary}\label{cor:kt-opt}
Let $S\subseteq S_{n}$ denote the set of optimal rankings for $\lop{A}$.
Then, an optimal solution of $\kt{A}{k^{*}}$ yields $\sigma_{1},\sigma_{2}\in S$ such that $C(\sigma_{1},\sigma_{2})$ is minimal and, equivalently, $D(\sigma_{1},\sigma_{2})$ is maximal with respect to $S$.
\end{corollary}
\begin{proof}
Note that every pair of rankings $\sigma_{1},\sigma_{2}\in S$ corresponds to at least one feasible solution of $\kt{A}{k^{*}}$, and vice-versa.
Hence, by the objective sense of~\eqref{eq:kt-obj} and Theorem~\ref{thm:kt-opt}, any optimal solution of $\kt{A}{k^{*}}$ yields $\sigma_{1},\sigma_{2}$ such that $C(\sigma_{1},\sigma_{2})$ is minimal with respect to $S$. 
Furthermore, since
\[
\abs{C(\sigma_{1},\sigma_{2})} = \binom{n}{2} - \abs{D(\sigma_{1},\sigma_{2})},
\]
it follows that $D(\sigma_{1},\sigma_{2})$ is maximal with respect to $S$.
\end{proof}

Let $\kappa(A)$ denote the maximal Kendall tau ranking distance between any two optimal solutions of $\lop{A}$.
Then, Corollary~\ref{cor:kt-opt} implies that the binary program $\kt{A}{k^{*}}$ can be used to compute the value of $\kappa(A)$ for any $\lop{A}$ with optimal value $k^{*}$.
Furthermore, by Proposition~\ref{prop:kt-opt}, the optima of $\kt{A}{k^{*}}$ yields two optimal solutions $X$ and $Y$ of $\lop{A}$ that are as far apart as possible with respect to the Kendall tau ranking distance.

It is worth  noting that since the input for $\kt{A}{k^{*}}$ is an optimal value for $\lop{A}$, the computation of $\kappa(A)$ is an NP-hard problem. 
This is not to say that the complexity of the program $\kt{A}{k^{*}}$ is necessarily NP-hard.
Indeed, if we assume that the value of $k^{*}$ is provided via an oracle machine, then the NP-hardness of $\kt{A}{k^{*}}$ is not immediately clear since there is no obvious way to apply the standard reductions argument.
\begin{example}\label{ex:kt_college_rankings}
Using the \emph{count} method in the \emph{SCIP} optimization suite~\cite{SCIP}, we enumerate all optima for the college ranking LOP from Example~\ref{ex:college-rankings}, as shown below:

\begin{minipage}{0.45\textwidth}
\begin{align*}
\sigma_{1} &= (10,7,8,5,1,6,9,2,4,3), \\
\sigma_{2} &= (10,7,5,8,1,6,9,2,4,3), \\
\sigma_{3} &= (10,7,5,1,8,6,9,2,4,3), \\
\end{align*}
\end{minipage}\hfill
\begin{minipage}{0.45\textwidth}
\begin{align*}
\sigma_{4} &= (10,7,8,5,1,6,9,2,3,4), \\
\sigma_{5} &= (10,7,5,8,1,6,9,2,3,4), \\
\sigma_{6} &= (10,7,5,1,8,6,9,2,3,4). \\
\end{align*}
\end{minipage}

Note that the six optima only vary in their relative rankings of Amherst, Carleton, Claremont, Haveford, and Swarthmore; thus, the relative rankings of all other colleges: Bowdoin, Middlebury, Ponoma, Wellsely, and Williams are firmly decided.

In Example~\ref{ex:college-rankings}, we saw that the degree of linearity $\lambda(A)$ is approximately equal to $0.75$, i.e., about $75\%$ of the data aligns with any of the optimal rankings shown above. 
Moreover, using \emph{CPLEX}, we can solve $\kt{A}{k^{*}}$ to find $\kappa(A)=3$, i.e., there is a maximum of three discordant pairs between any two optimal rankings of $\lop{A}$, which is readily verified by observing the optimal rankings shown above.
Note that $\lambda(A)$ and $\kappa(A)$ can be used to analyze the rankability of the college rankings data.
Specifically, we know that approximately $75\%$ of our data agrees with any optimal ranking of $\lop{A}$ and that any two optimal rankings can have at most $3$ discordant pairs, i.e., any two optimal rankings will agree with at least $42/45\approx 93\%$ of the relative college rankings. 
It is important to note that this information can be obtained solely from $\lambda(A)$ and $\kappa(A)$, and we do not require all optimal solutions which is impractical in many applications. 
\end{example}
\subsubsection{Optimally Valid Inequalities}
The general solution process of binary programs in well-established solvers like \emph{CPLEX}, \emph{GUROBI}~\cite{gurobi}, or 
\emph{SCIP}, typically involves a partial enumeration scheme known as the branch-and-bound approach. 
Note that, given $n$ binary decision variables, a total enumeration tree would have $2^{n+1}-1$ nodes, where each node corresponds to the given binary program with some of its decision variables fixed to either zero or one. 
Since this is exponential in the size of the input, straight enumeration is not useful in practice. 
In the branch-and-bound approach, the linear programming relaxation is used to prune off branches of the tree to make the search more efficient in practice.

Sometimes additional information that is known about the optimal solutions of the binary program can be expressed as a linear inequality. 
If this linear inequality is valid only for the optimal solutions of the considered binary program, then we say that it is an \emph{optimally valid inequality}. 
Adding such inequalities to the original binary program may significantly change the underlying polytope as it cuts off all non-optimal solutions for which the inequality is not valid. 
This in turn may lead to a smaller branch-and-bound tree, and a faster solution process in practice.

Propositions~\ref{prop:opt_ineq1}--\ref{prop:opt_ineq2} show two optimally valid inequalities for $\kt{A}{k^{*}}$.
We have not verified the impact of these optimally valid inequalities from a computational perspective, but believe they are worth mentioning and investigating in this way. 
For small $n$, these inequalities can be explicitly added to the program $\kt{A}{k^{*}}$.
Otherwise, more advanced techniques such as row generation can be used~\cite{Grotschel1984}.
\begin{proposition}\label{prop:opt_ineq1}
For each $i,j\in[n]$ such that $i\neq j$, the inequality
\begin{equation}\label{eq:opt_ineq1}
z_{ij}+z_{ji}\leq 1
\end{equation}
is optimally valid for $\kt{A}{k^{*}}$.
\end{proposition}
\begin{proof}
By Proposition~\ref{prop:kt-opt}, an optimal solution of $\kt{A}{k^{*}}$ has $z_{ij}=1$ for some $i,j\in[n]$ such that $i\neq j$, if and only if $x_{ij}=y_{ij}=1$. 
Since $x$ and $y$ correspond to optimal solutions of $\lop{A}$, it follows from constraint~\eqref{eq:lop-const1} that $z_{ij}$ and $z_{ji}$ cannot both equal $1$, which implies that~\eqref{eq:opt_ineq1} holds for all optimal solutions of $\kt{A}{k^{*}}$.

Moreover, note that feasible solutions of $\kt{A}{k^{*}}$ that are not optimal can have both $z_{ij}$ and $z_{ji}$ equal to $1$ since the objective function in~\eqref{eq:kt-obj} is not minimized. 
Thus, for feasible solutions of $\kt{A}{k^{*}}$ that are not optimal,~\eqref{eq:opt_ineq1} does not hold for all $i,j\in[n]$ such that $i\neq j$.
\end{proof}
\begin{proposition}\label{prop:opt_ineq2}
For each $i,j,k\in[n]$ such that $i<j$,  $i<k$, and $j\neq k$, the inequality
\begin{equation}\label{eq:opt_ineq2}
z_{ij} + z_{jk} + z_{ki} \leq 2
\end{equation}
is optimally valid for $\kt{A}{k^{*}}$.
\end{proposition}
\begin{proof}
By Proposition~\ref{prop:kt-opt}, an optimal solution of $\kt{A}{k^{*}}$ has $z_{ij}=1$ for some $i,j\in[n]$ such that $i\neq j$, if and only if $x_{ij}=y_{ij}=1$.
Since $x$ and $y$ correspond to optimal solutions of $\lop{A}$, it follows from constraint~\eqref{eq:lop-const2} that
\[
z_{ij} + z_{jk} + z_{ki} \leq 2,
\]
which implies that~\eqref{eq:opt_ineq2} holds for all optimal solutions of $\kt{A}{k^{*}}$. 

Moreover, note that feasible solutions of $\kt{A}{k^{*}}$ that are not optimal can have all $z_{ij}$, $z_{jk}$, and $z_{ki}$ equal to $1$ since the objective function in~\eqref{eq:kt-obj} is not minimized.
Thus, for feasible solutions of $\kt{A}{k^{*}}$ that are not optimal,~\eqref{eq:opt_ineq2} does not hold for all $i,j,k\in[n]$ such that $i<j$, $i<k$ and $j\neq k$. 
\end{proof}
\section{Rankability in the NFL}\label{sec:rank-nfl}
Rankings are often used in sports to determine a linear ordering of the teams. 
For instance, the BCS College Football Rankings is used to determine which teams get to play in the post-season.
These rankings are created by human polls and an aggregation of six computer rankings~\cite{Langville2012}.
In the NFL, the playoff teams are determined by division and win-loss records; however, rankings are still important in the betting world~\cite{Langville2012}.
In this section, we use the LOP to analyze the rankability of NFL game data, which was accessed from Kaggle~\cite{spreadspoke}.
In particular, we show how the degree of linearity and the maximal Kendall tau distance can be used to analyze the hindsight and foresight accuracy of rankings, which are two common strategies for comparing the relevance of sport rankings~\cite{Langville2012}.

\subsection{Hindsight Accuracy}
While rankings can be chosen at random, they are usually formed from game data.
For instance, the Massey method uses point-score data and the Colley method uses win-loss data to rate teams~\cite{Langville2012}. 
The \emph{hindsight accuracy} of a ranking is the percentage of games, from the data used to form that ranking, for which it correctly predicts the outcome.
For instance, suppose you obtain a ranking $\sigma\in S_{n}$ of $n\geq 2$ teams from the regular season game data for a particular year.
Then, the hindsight accuracy of $\sigma$ is the percentage of games during the regular season for which the higher ranked team in $\sigma$ won the game. 

Now, define $A=[a_{ij}]$, where $a_{ij}$ is the number of games team $i$ won against team $j$ and $0.5$ is awarded to both teams in the case of a tie. 
Then, the following result provides a non-trivial bound on the hindsight accuracy of any ranking $\sigma\in S_{n}$.
\begin{proposition}\label{prop:hind_acc}
The degree of linearity of $A$ is an upper bound on the hindsight accuracy of any ranking in $S_{n}$. 
Moreover, the rankings in $S_{n}$ that have maximum hindsight accuracy correspond to the optimal rankings of $\lop{A}$.
\end{proposition}
\begin{proof}
Let $t$ denote the number of games that ended in a tie and define $c(\sigma)$ to be the number of games correctly predicted by $\sigma\in S_{n}$.
For each $\sigma\in S_{n}$, the corresponding symmetric re-ordering of $A$ has an upper-triangular sum equal to $c(\sigma) + 0.5t$.
Thus, the optimal value of $\lop{A}$ satisfies
\[
k^{*} = \max_{\sigma\in S_{n}}c(\sigma) + 0.5 t.
\]
Let $m$ denote the total number of games played, then the degree of linearity of $A$ satisfies
\[
\lambda(A) = \frac{k^{*}}{\sum_{i\neq j}a_{ij}} = \max_{\sigma\in S_{n}}\frac{c(\sigma) + 0.5t}{m} \geq \max_{\sigma\in S_{n}}\frac{c(\sigma)}{m},
\]
and the result follows.
\end{proof}
\begin{example}\label{ex:nfl_hindsight}
From 1970 to 2019, the matrix $A$ is formed using the NFL regular season game data. 
Then, using the \emph{CPLEX} optimization suite, we compute the optimal value and an optimal ranking of $\lop{A}$. 
In Figure~\ref{fig:nfl_hindsight}, we report the degree of linearity of $A$ and the hindsight accuracy of the given optimal ranking of $\lop{A}$.
Moreover, for comparison, we include the hindsight accuracy of the Massey and Colley rankings. 
\begin{figure}[ht]
\centering
\includegraphics[width=1.0\textwidth]{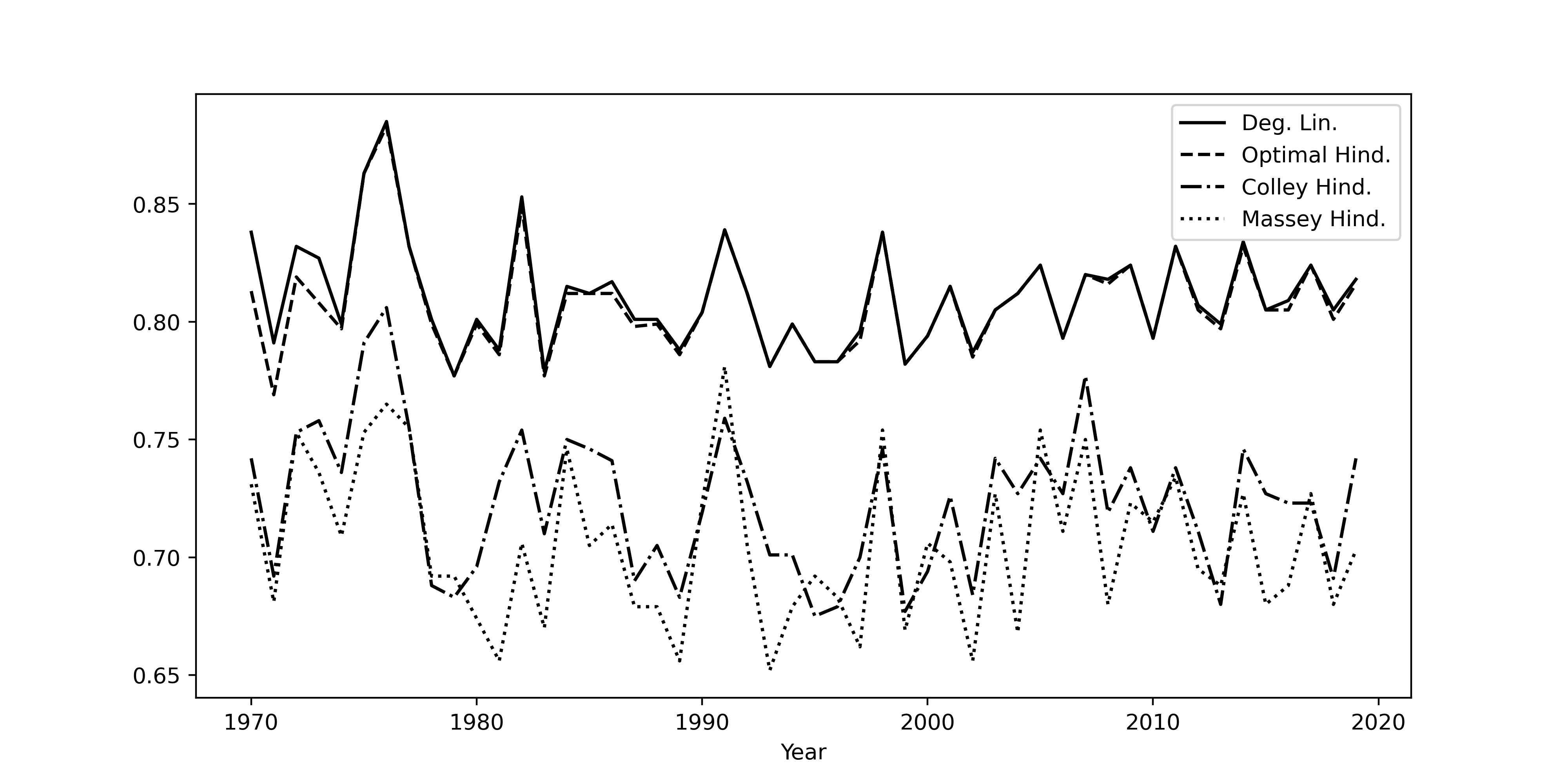}
\caption{The degree of linearity and the hindsight accuracy of NFL rankings.}
\label{fig:nfl_hindsight}
\end{figure}

As expected from Proposition~\ref{prop:hind_acc}, the degree of linearity serves as an upper bound for the hindsight accuracy of all rankings. 
Moreover, the hindsight accuracy of the optimal ranking is nearly identical to the degree of linearity, the only discrepancy is the result of games that ended in a tie. 
Finally, it is worth noting that the hindsight accuracy of the Massey and Colley rankings also correlate strongly with the degree of linearity; consider the correlation matrix in Figure~\ref{fig:nfl_hindsight_corr}.
\begin{figure}[ht]
\centering
\includegraphics[width=0.55\textwidth]{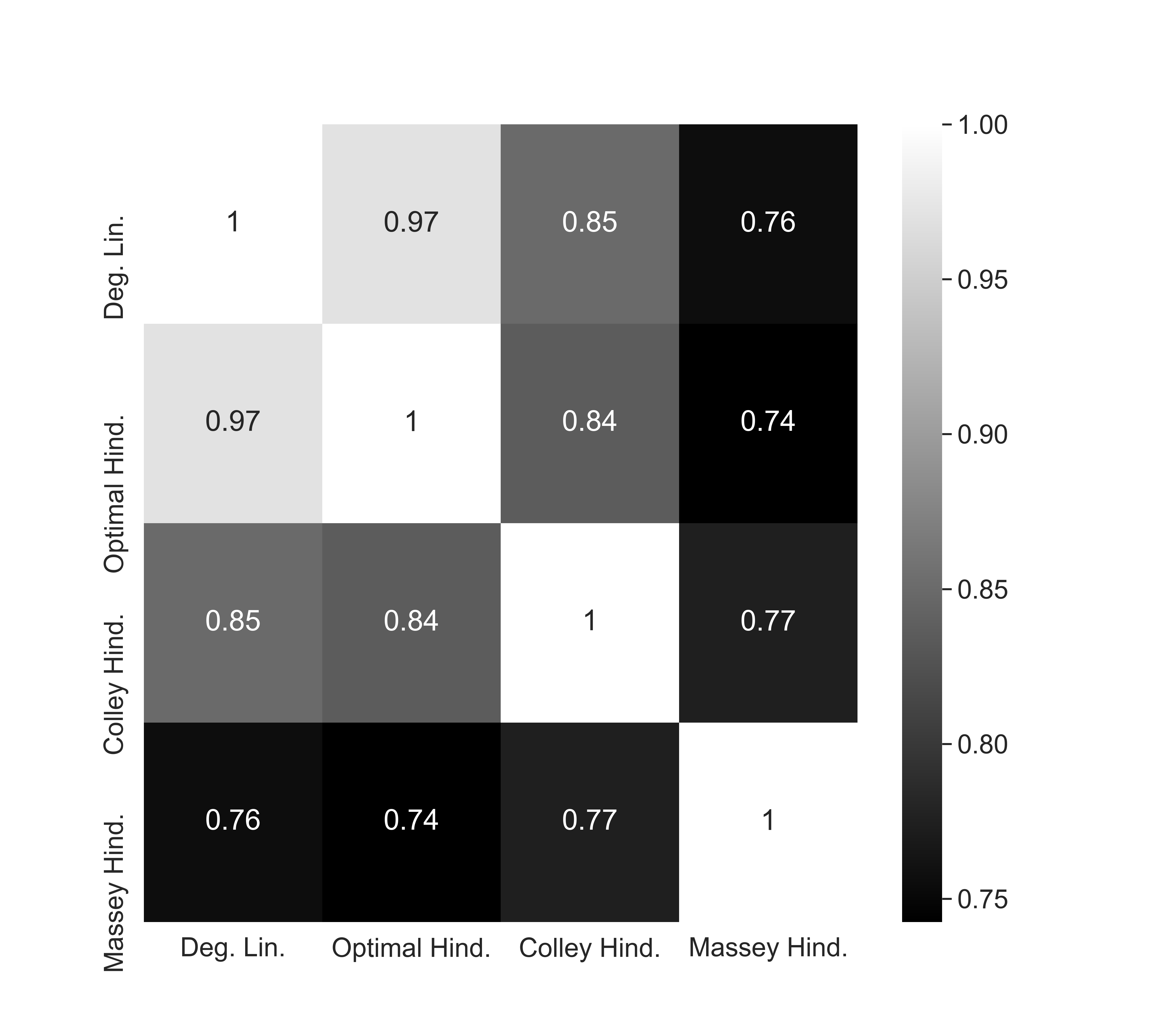}
\caption{Correlation between the degree of linearity and the hindsight accuracy of NFL rankings.}
\label{fig:nfl_hindsight_corr}
\end{figure}
\end{example}
\subsection{Foresight Accuracy}
The \emph{foresight accuracy} of a ranking is the percentage of future games that ranking can predict.
For instance, suppose $\sigma\in S_{n}$ is a ranking of $n\geq 2$ teams formed from the regular season game data for a particular year. 
The foresight accuracy of $\sigma$ would be the percentage of playoff games for which the higher ranked team in $\sigma$ won the game.
In general, there is no relationship between the hindsight accuracy and foresight accuracy of a ranking.
As an example, note that the correlation between the hindsight and foresight accuracy of the Colley and Massey rankings from 1970 to 2019 in the NFL is $0.0451$ and $-0.1368$, respectively.
This lack of correlation can be partially explained by the structure of the NFL playoffs, that is, only certain teams make it to the playoffs, and not all of those teams play each other as is done in round-robin style tournaments.
\begin{example}\label{ex:forw_pred_diff}
From 1970 to 2019, the matrix $A$ from Proposition~\ref{prop:hind_acc} is formed using the regular season game data. 
Note that any two optimal rankings of $\lop{A}$ have the same hindsight accuracy,  yet they can have vastly different foresight accuracy. 
In Figure~\ref{fig:forw_pred_diff}, we show the absolute difference between the foresight accuracy of two optimal rankings of $\lop{A}$.
In particular, the two optimal rankings are computed via $\kt{A}{k^{*}}$ to be as far apart as possible with respect to the Kendall tau ranking distance. 
Note that these two rankings are not necessarily the most diverse among the playoff teams; hence, they may not maximize the absolute difference between the foresight accuracy of two optimal rankings of $\lop{A}$.
\begin{figure}[ht]
\centering
\includegraphics[width=0.75\textwidth]{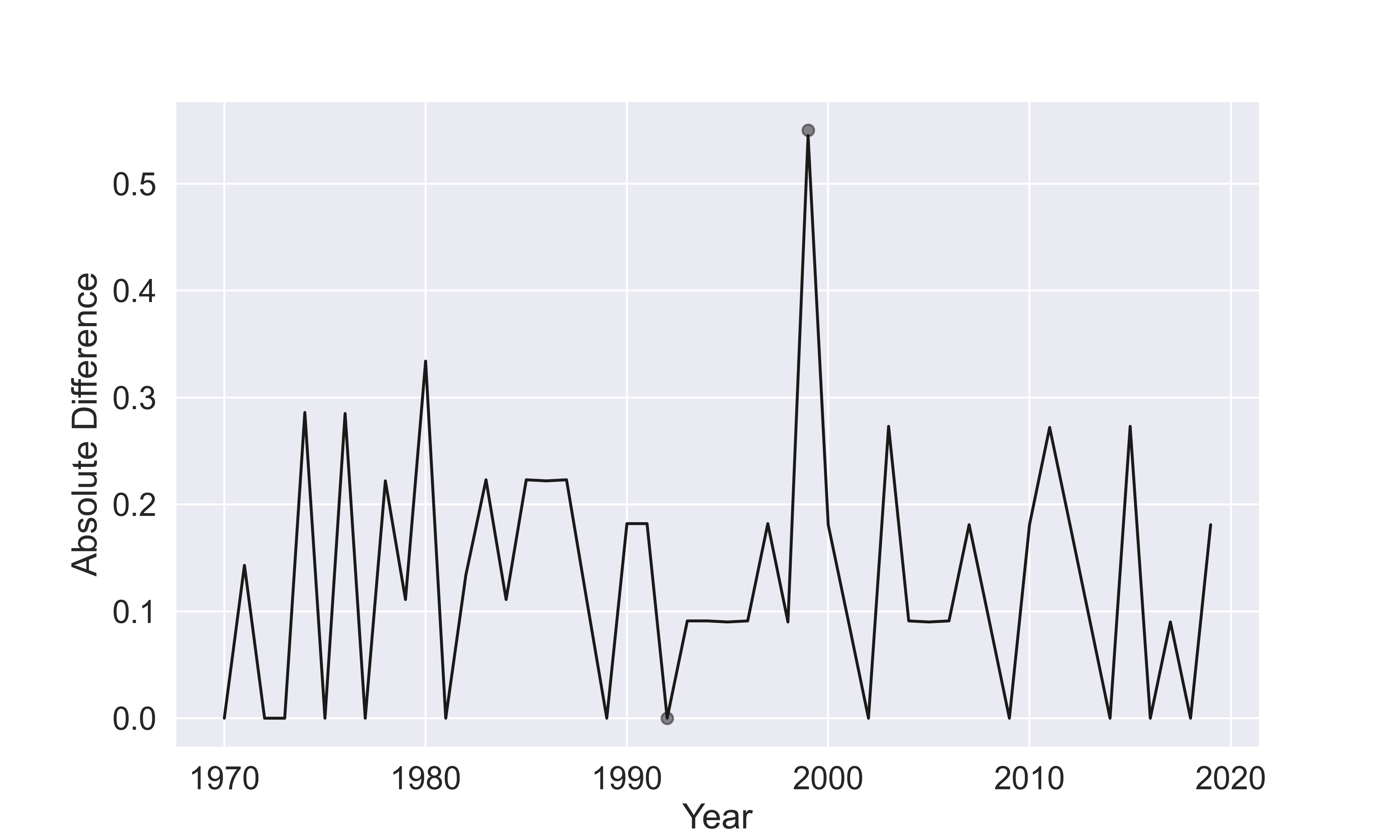}
\caption{The absolute difference between the foresight accuracy of two optimal rankings.}
\label{fig:forw_pred_diff}
\end{figure}
\end{example}

Nonetheless, Figure~\ref{fig:forw_pred_diff} still leads to an interesting analysis, especially if we consider the two extreme points: $(1992,0)$ and $(1999,0.55)$.
In Figure~\ref{fig:nfl_opt_rankings}, the playoff teams in two optimal NFL rankings computed via $\kt{A}{k^{*}}$ for the years 1992 and 1999 are shown.
The diversity between these rankings is displayed using dashed lines to illustrate how the ranking of an individual team changes.
\begin{figure}[ht]
\centering
\begin{tabular}{cc}
\includegraphics[width=0.40\textwidth]{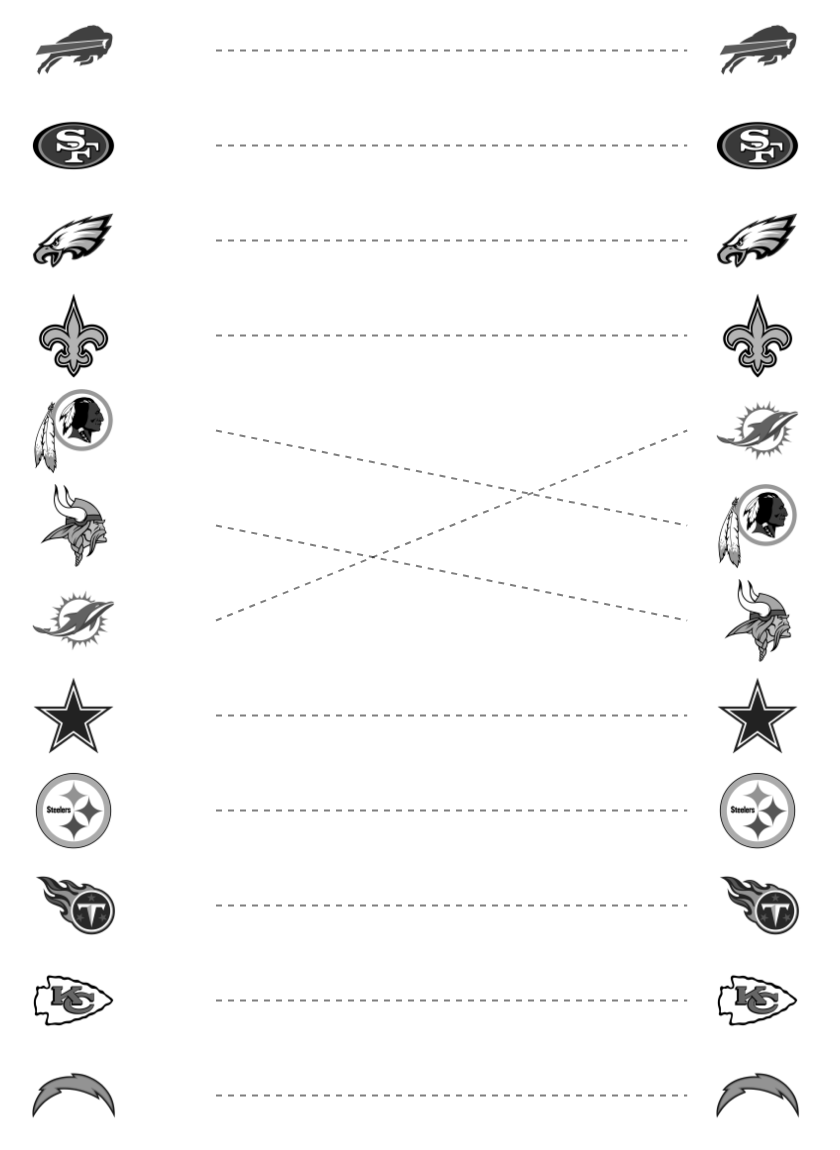}
&
\includegraphics[width=0.40\textwidth]{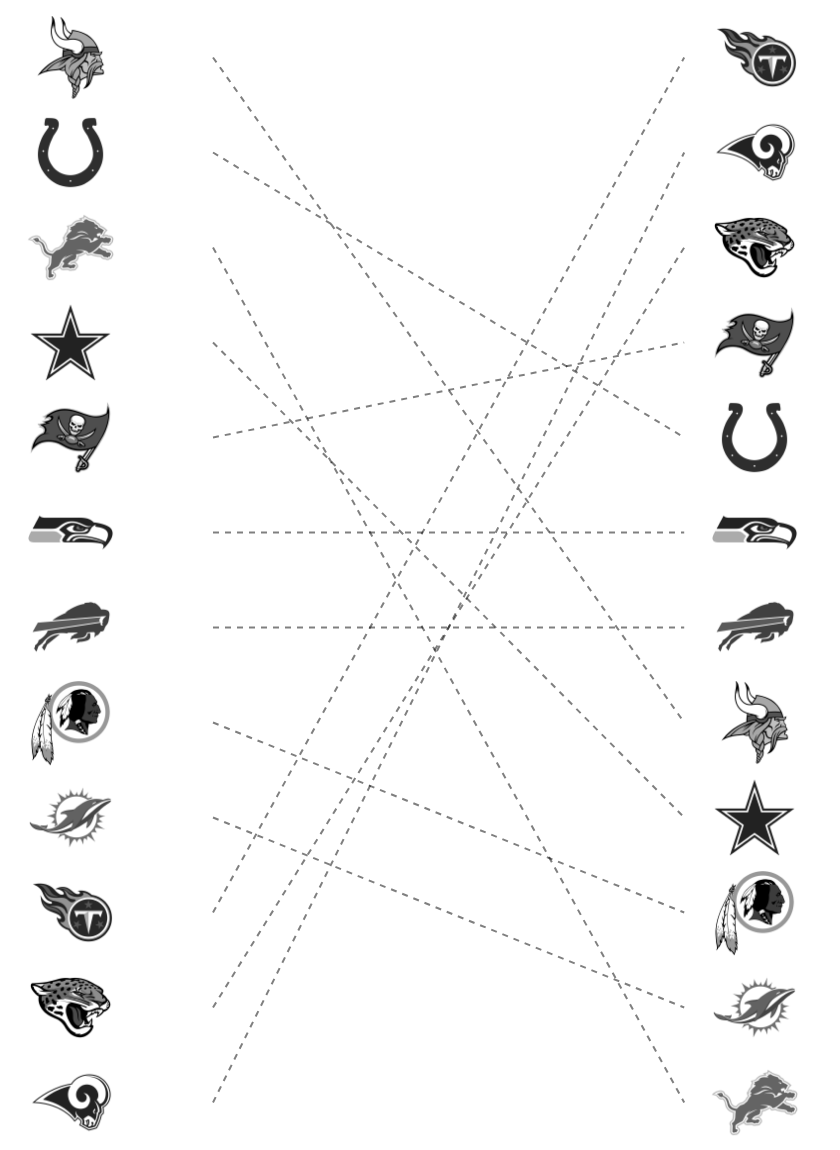} \\
\textbf{1992} & \textbf{1999} \\
\end{tabular}
\caption{Diversity among NFL playoff teams in two optimal rankings for years 1992 and 1999.}
\label{fig:nfl_opt_rankings}
\end{figure}
\begin{example}
In 1992, the Kendall Tau ranking distance among the playoff teams is equal to $2$, which corresponds to the relative changes in the rankings among the Redskins, Vikings, and Dolphins.
Moreover, the only playoff game among these three teams is between the Vikings and Redskins, see Figure~\ref{fig:nfl_playoff1992}.
Since both optimal rankings in 1992 have the Redskins ranked higher than the Vikings, they both correctly predict the outcome of that game.
Moreover, since all other games in the playoff progression are between teams whose relative ranking is the same in both rankings in 1992, it follows that these rankings have the same prediction for each game.
\begin{figure}[ht]
\centering
\includegraphics[trim={0 6cm 0 0},clip,width=1.0\textwidth]{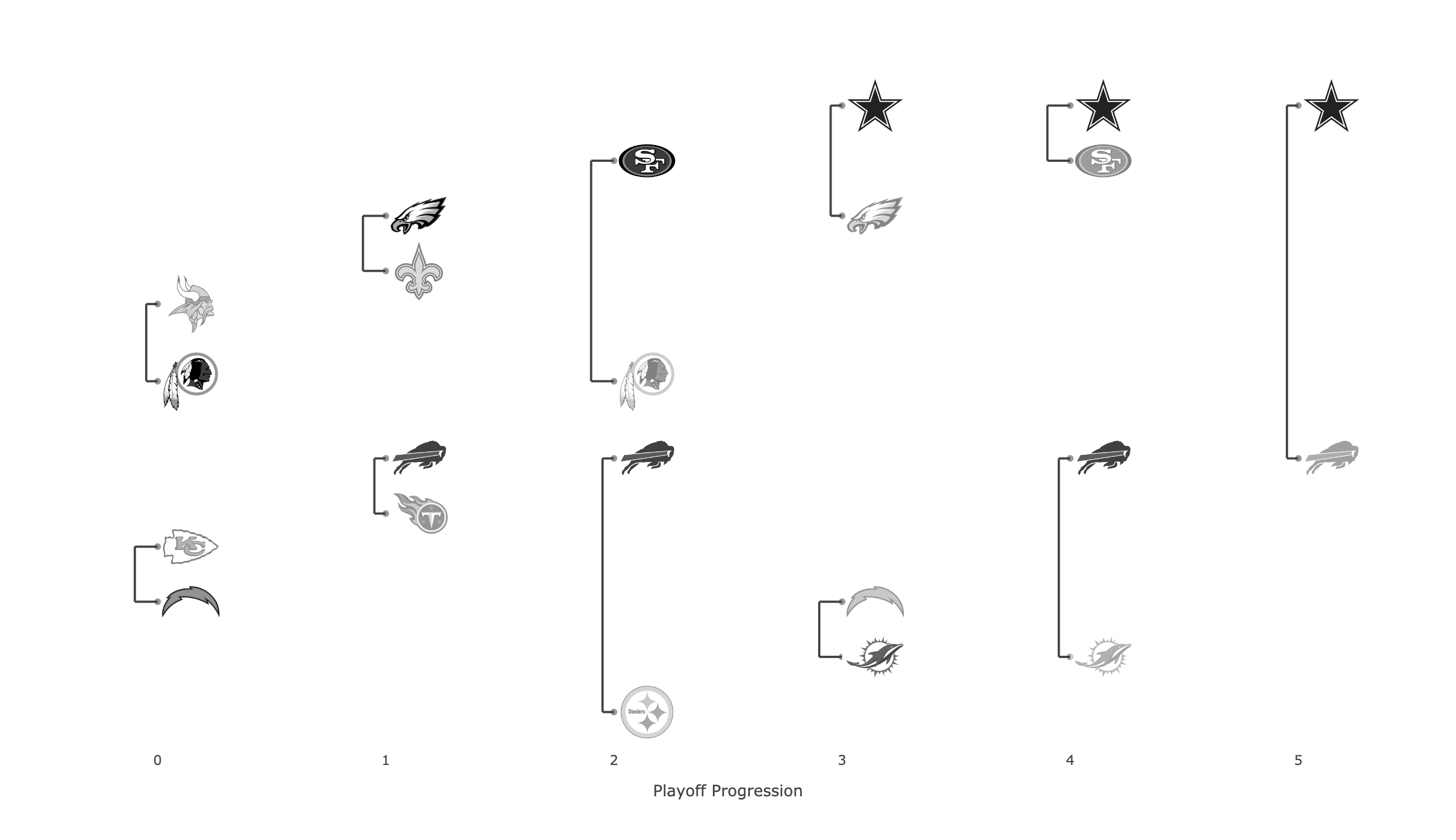}
\caption{NFL playoff progression in 1992; the opaque team lost the indicated matchup.}
\label{fig:nfl_playoff1992}
\end{figure}

It is worth emphasizing that both optimal rankings from 1992 have a hindsight accuracy of $81\%$ and a foresight accuracy of $64\%$. 
Also, these optimal rankings did not correctly predict any of the three games that the Dallas Cowboys played in.
Interestingly,  the betting odds favored the Cowboys in two of those three games, including the Super Bowl game where the Cowboys were favored to win over the Bills with a point-spread of $-6.5$; moreover, the Cowboys covered the spread winning $52$ to $17$. 
It seems that the betting experts were able to see something that these optimal rankings based solely on win-loss data did not see.
\end{example}

\begin{example}
In 1999, there is much more diversity in the relative rankings of the playoff teams.
For instance, the Super Bowl winner: St. Louis Rams, can be considered the 2nd best team in the right-sided optimal ranking and the worst team in the left-sided optimal ranking as shown in Figure~\ref{fig:nfl_opt_rankings}.
It seems that the right-sided ranking is more inline with how the betting experts viewed these teams since the Rams were favored in all three of their games, and they covered the spread in two of those three games.
Moreover, the right-sided ranking correctly predicts the outcome of $82\%$ of the playoff games, whereas the left-sided optimal ranking only correctly predicts the outcome of $27\%$ of the playoff games, see Figure~\ref{fig:nfl_playoff1999}.
\begin{figure}[ht]
\centering
\includegraphics[trim={0 6cm 0 0},clip,width=1.0\textwidth]{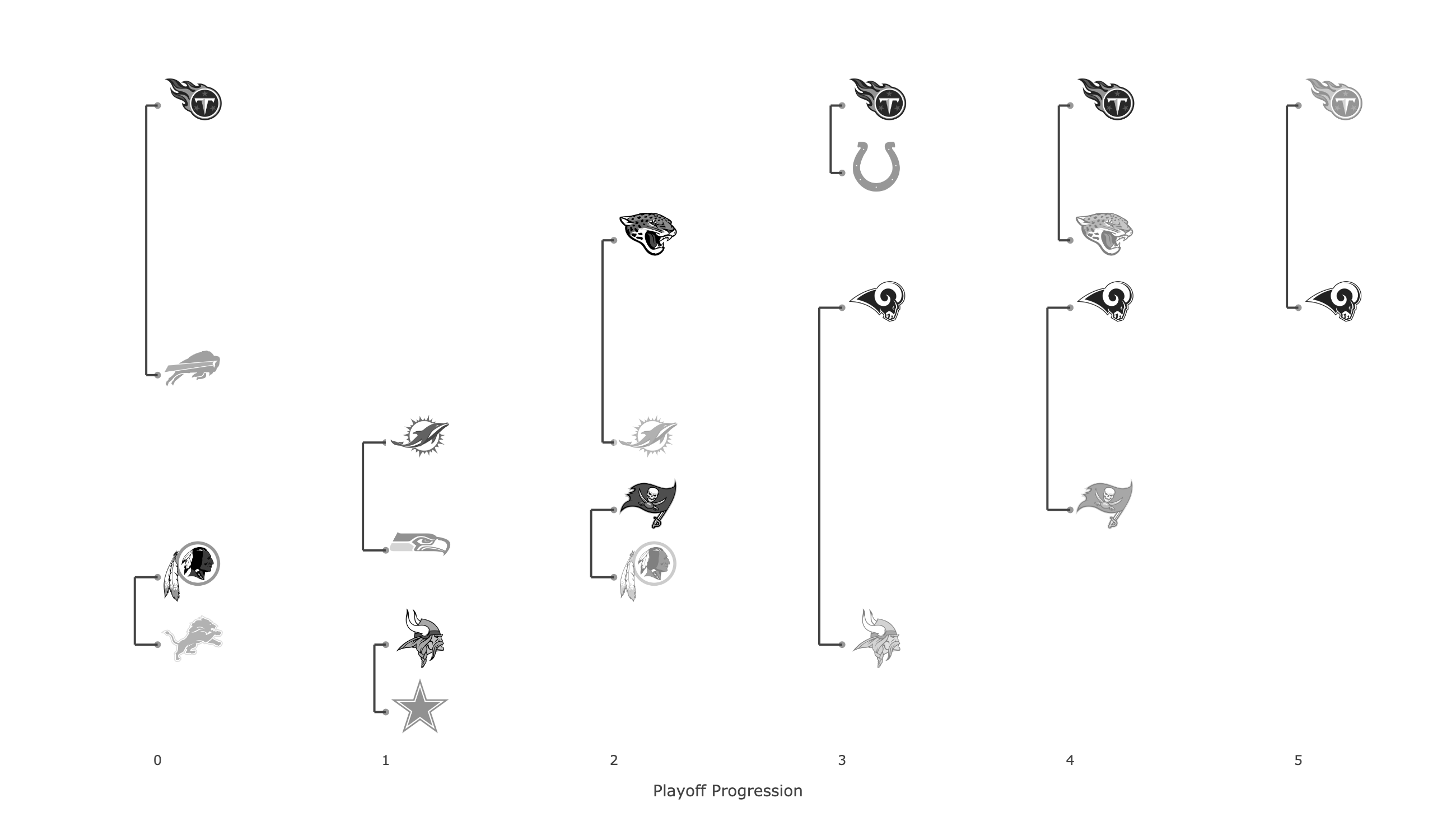}
\caption{NFL playoff progression in 1999; the opaque team lost the indicated matchup.}
\label{fig:nfl_playoff1999}
\end{figure} 
\end{example}
\section{Conclusion}\label{sec:conclusion}
The linear ordering problem (LOP) provides novel insight into the rankability of a dataset.
In particular, the degree of linearity can be used as a non-trivial upper bound on the hindsight accuracy of all rankings. 
In Figure~\ref{fig:nfl_hindsight}, we saw that the hindsight accuracy of the Massey and Colley NFL rankings were usually around $5\%$ to $10\%$ below the degree of linearity. 

Also, the binary program $\kt{A}{k^{*}}$ can be used to measure the maximal Kendall tau ranking distance between any two optimal rankings of $\lop{A}$. 
In Figures~\ref{fig:forw_pred_diff} and~\ref{fig:nfl_opt_rankings}, we saw that this binary program can be used to identify optimal rankings that have vastly different foresight accuracy.
As noted in Example~\ref{ex:forw_pred_diff}, this binary program does not necessarily maximize the diversity among playoff teams. 
For this reason, as future research, it would be worthwhile to consider adaptations to $\kt{A}{k^{*}}$ that seek to maximize diversity among playoff teams.
It may even be worth considering maximizing the diversity among playoff teams that are from the same conference, thus optimizing the likelihood of those teams playing each other in the playoffs.
Of course, this motivation can be generalized to adapting $\kt{A}{k^{*}}$ to optimize the diversity in the optimal rankings of $\lop{A}$ among a subset of vertices. 

We conclude by noting that the NFL data in Section~\ref{sec:rank-nfl}, required us to solve instances of $\lop{A}$ and $\kt{A}{k^{*}}$ to optimality for $n\geq 30$. 
Also, instances of the program from~\cite{Kondo2014}, which is similar to $\kt{A}{k^{*}}$, have been solved to optimality for $n\geq 100$.
While the complexity of $\lop{A}$ and $\kt{A}{k^{*}}$ suggest computational limitations, more advanced techniques such as column and row generation combined with heuristics will likely lead to improved scalability of our work~\cite{Barnhart1998,Chanas1996,Grotschel1984}.
For future research, we are interested in pursuing questions that arise from the computational challenges of measuring rankability for larger data sets.
\section*{Acknowledgments} 
The authors wish to acknowledge three anonymous referees whose thoughtful comments and suggestions greatly improved this article.

\providecommand{\href}[2]{#2}
\providecommand{\arxiv}[1]{\href{http://arxiv.org/abs/#1}{arXiv:#1}}
\providecommand{\url}[1]{\texttt{#1}}
\providecommand{\urlprefix}{URL }

\medskip
\medskip


\providecommand{\href}[2]{#2}
\providecommand{\arxiv}[1]{\href{http://arxiv.org/abs/#1}{arXiv:#1}}
\providecommand{\url}[1]{\texttt{#1}}
\providecommand{\urlprefix}{URL }
\begin{thebibliography}{10}

\bibitem{Anderson2019}
\newblock P.~E. Anderson, T.~P. Chartier and A.~N. Langville,
\newblock The rankability of data,
\newblock \emph{SIAM J. Math. Data Sci.}, \textbf{1} (2019), 121--143.

\bibitem{Anderson2021}
\newblock P.~E. Anderson, T.~P. Chartier, A.~N. Langville and K.~E.
  Pedings-Behling,
\newblock The rankability of weighted data from pairwise comparisons,
\newblock Foundations of Data Science, 2021.

\bibitem{Barnhart1998}
\newblock C.~Barnhart, E.~L. Johnson, G.~L. Nemhauser, M.~W.~P. Savelsbergh and
  P.~H. Vance,
\newblock Branch-and-price: column generation for solving huge integer
  programs,
\newblock \emph{Operations Research}, \textbf{46} (1998), 316--329.

\bibitem{Cameron2020_SR}
\newblock T.~R. Cameron, A.~N. Langville and H.~C. Smith,
\newblock On the graph {L}aplacian and the rankability of data,
\newblock \emph{Linear Algebra Appl.}, \textbf{588} (2020), 81--100.

\bibitem{Cameron2020_KT}
\newblock T.~R. Cameron, J.~Pulaj and S.~Charmot,
\newblock Diameter polytopes of feasible binary programs,
\newblock arXiv preprint arXiv:2008.06844, 2020.

\bibitem{Chanas1996}
\newblock S.~Chanas and P.~Kobylanski,
\newblock A new heuristic algorithm solving the linear ordering problem,
\newblock \emph{Computational Optimization and Applications}, \textbf{6}
  (1996), 191--205.

\bibitem{Chartier2014}
\newblock T.~Chartier and J.~Peachey,
\newblock Reverse engineering college rankings,
\newblock \emph{Math Horizons}, \textbf{22} (2014), 5--7.

\bibitem{Chau2020}
\newblock S.~L. Chau, J.~Gonz{\'a}lez and D.~Sejdinovic,
\newblock Learning inconsistent preferences with kernal methods,
\newblock arXiv preprint arXiv:2006.03847, 2020.

\bibitem{Chiarni2004}
\newblock B.~H. Chiarni,
\newblock \emph{New Algorithm for the triangulation of Input-Output Tables and
  the Linear Ordering Problem},
\newblock Master's thesis, University of Florida, 2004.

\bibitem{cplex2019}
\newblock CPLEX,
\newblock \emph{IBM ILOG CPLEX 12.9 User’s Manual},
\newblock IBM ILOG CPLEX Division, Incline Village, NV, 2019.

\bibitem{Dantzig1991}
\newblock G.~B. Dantzig,
\newblock \emph{Linear Programming and Extensions},
\newblock Princeton University Press, Princeton, NJ, 1991.

\bibitem{SCIP}
\newblock G.~Gamrath, D.~Anderson, K.~Bestuzheva, W.-K. Chen, L.~Eifler,
  M.~Gasse, P.~Gemander, A.~Gleixner, L.~Gottwald, K.~Halbig, G.~Hendel,
  C.~Hojny, T.~Koch, P.~L. Bodic, S.~J. Maher, F.~Matter, M.~Miltenberger,
  E.~M{\"u}hmer, B.~M{\"u}ller, M.~Pfetsch, F.~Schl{\"o}sser, F.~Serrano,
  Y.~Shinano, C.~Tawfik, S.~Vigerske, F.~Wegscheider, D.~Weninger and
  J.~Witzig,
\newblock \emph{{The SCIP Optimization Suite 7.0}},
\newblock ZIB-Report 20-10, Zuse Institute Berlin, 2020,
\newblock \urlprefix\url{http://nbn-resolving.de/urn:nbn:de:0297-zib-78023}.

\bibitem{Garey1979}
\newblock M.~R. Garey and D.~S. Johnson,
\newblock \emph{Computers and intractability: A guide to the theory of
  NP-completeness},
\newblock Freeman, New York, NY, 1979.

\bibitem{Glover1974}
\newblock F.~Glover, T.~Klastorin and D.~Klingman,
\newblock Optimal weighted ancestry relationships,
\newblock \emph{Management Science}, \textbf{20} (1974), 1190--1193.

\bibitem{Glover1998}
\newblock F.~Glover, C.~C. Kuo and K.~S. Dhir,
\newblock Heuristic algorithms for the maximum diversity problem,
\newblock \emph{J. Inform. Optim. Sci.}, \textbf{19} (1998), 109--132.

\bibitem{Grotschel1984}
\newblock M.~Gr{\"o}tschel, M.~J{\"u}nger and G.~Reinelt,
\newblock A cutting plane algorithm for the linear ordering problem,
\newblock \emph{Operations Research}, \textbf{32} (1984), 1195--1220.

\bibitem{Grotschel1985:2}
\newblock M.~Gr{\"o}tschel, M.~J{\"u}nger and G.~Reinelt,
\newblock Facets of the linear ordering problem,
\newblock \emph{Mathematical Programming}, \textbf{33} (1985), 43--60.

\bibitem{Grotschel1985:1}
\newblock M.~Gr{\"o}tschel, M.~J{\"u}nger and G.~Reinelt,
\newblock On the acyclic subgraph polytope,
\newblock \emph{Mathematical Programming}, \textbf{33} (1985), 28--42.

\bibitem{Grotschel1985:3}
\newblock M.~Gr{\"o}tschel and M.~W. Padberg,
\newblock Polyhedral theory,
\newblock in \emph{The {T}raveling salesman problem},
\newblock John Wiley \& Sons, Philadelphia, PA, 1985,
\newblock chapter~8, 251--302.

\bibitem{gurobi}
\newblock Gurobi,
\newblock Gurobi optimizer reference manual, 2020,
\newblock \urlprefix\url{http://www.gurobi.com}.

\bibitem{Kendall1938}
\newblock M.~G. Kendall,
\newblock A new measure of rank correlation,
\newblock \emph{Biometrika}, \textbf{30} (1938), 81--93.

\bibitem{Khachiyan1979}
\newblock L.~G. Khachiyan,
\newblock A polynomial algorithm in linear programming,
\newblock \emph{Dokl. Akad. Nauk SSSR}, \textbf{244} (1979), 1093--1096.

\bibitem{Kondo2014}
\newblock Y.~Kondo,
\newblock Triangulation of input–output tables based on mixed integer
  programs for inter-temporal and inter-regional comparison of production
  structures,
\newblock \emph{J. Econ. Struct.}, \textbf{3} (2014), 1--19.

\bibitem{Kuo1993}
\newblock C.~C. Kuo,
\newblock Analyzing and modeling the maximum diversity problem by zero-one
  programming,
\newblock \emph{Decision Sci.}, \textbf{24} (1993), 1171--1185.

\bibitem{Langville2012}
\newblock A.~N. Langville and C.~D. Meyer,
\newblock \emph{Who's {\#}1?: The science of rating and ranking},
\newblock Princeton University Press, Princeton, NJ, 2012.

\bibitem{Lee2004}
\newblock J.~Lee,
\newblock \emph{A First Course in Combinatorial Optimization},
\newblock Cambridge University Press, Cambridge, England, 2004.

\bibitem{Leontiff1986}
\newblock W.~Leontiff,
\newblock \emph{Input-Output Economics},
\newblock Oxford University Press, New York, USA, 1986.

\bibitem{Marti2011}
\newblock R.~Mart{\'i} and G.~Reinelt,
\newblock \emph{The {L}inear {O}rdering {P}roblem},
\newblock Springer-Verlag, Berlin, Germany, 2011.

\bibitem{Petit2019}
\newblock T.~Petit and A.~C. Trapp,
\newblock Enriching solutions to combinatorial problems via solution
  engineering,
\newblock \emph{INFORMS J. Comput.}, \textbf{31} (2019), 429--444.

\bibitem{Reinelt1993}
\newblock G.~Reinelt,
\newblock A note on small linear ordering polytopes,
\newblock \emph{Discrete Comput. Geom}, \textbf{10} (1993), 67--78.

\bibitem{spreadspoke}
\newblock Spreadspoke,
\newblock Nfl scores and betting data,
\newblock https://www.kaggle.com/tobycrabtree/nfl-scores-and-betting-data,
  2020.

\bibitem{Tsai2008}
\newblock J.-F. Tsai, M.-H. Lin and Y.-C. Hu,
\newblock Finding multiple solutions to general integer linear programs,
\newblock \emph{European Journal of Operational Research}, \textbf{184} (2008),
  802--809.

\end{thebibliography}
\end{document}